\documentclass{amsart}

\usepackage[pagebackref=false,colorlinks=true, pdfstartview=FitV, linkcolor=blue,citecolor=red, urlcolor=blue]{hyperref}

\usepackage{mathrsfs}

\newtheorem{theorem}{Theorem}[subsection]

\newtheorem{lemma}[theorem]{Lemma}
\newtheorem{proposition}[theorem]{Proposition}

\theoremstyle{definition}

\theoremstyle{definition}

\theoremstyle{remark}

\newcommand\C{{\bf C}}

\newcommand\M{{\mathscr M}}
\newcommand\B{{\mathscr B}}

\def\n{\ensuremath {\mathfrak n}}
\def\m{\ensuremath {\mathfrak m}}
\def\H{\ensuremath {{\bf H}}}
\def\P{\ensuremath {{\bf P}}}

\def\Q{\ensuremath {{\bf Q}}}
\def\T{\ensuremath {{\bf T}}}
\def\O{\ensuremath {{\mathscr O}}}
\def\D{\ensuremath {{\mathscr D}}}
\def\R{\ensuremath {{\mathscr R}}}
\def\V{\ensuremath {{\mathscr V}}}
\def\CC{\ensuremath {{\mathscr C}}}
\def\SS{\ensuremath {{\mathscr S}}}
\def\X{\ensuremath {{\mathcal X}}}
\newcommand\be{\begin{equation}}
\newcommand\ee{\end{equation}}

\keywords{Bianchi modular form; modular symbol; Manin symbol; Hecke operator; Fourier expansion.}


\begin{document}

%

%

\title{Universal Fourier expansions of Bianchi modular forms
}

\author{Tian An Wong}

\address{Department of Mathematics,\\
The University of British Columbia\\
Room 121, 1984 Mathematics Road\\
Vancouver, BC\\
Canada V6T 1Z2}
\email{wongtianan@math.ubc.ca}

\maketitle

\begin{abstract}
We generalize Merel's work on universal Fourier expansions to Bianchi modular forms over Euclidean imaginary quadratic fields, under the assumption of the nondegeneracy of a pairing between Bianchi modular forms and Bianchi modular symbols. Among the key inputs is a computation of the action of Hecke operators on Manin symbols, and building upon the Heilbronn-Merel matrices constructed by Mohamed. 
\end{abstract}

\section{Introduction}

Modular symbols provide a formalism enabling computations with modular forms. Merel's universal Fourier expansion \cite{merel} provided a way to more efficiently compute modular symbols using the basis of Manin symbols \cite{stein}. Before Merel, computing the Hecke action required converting Manin symbols to modular symbols, applying the Hecke action, and then converting back.

In this paper, we will be interested in Bianchi modular symbols, that is, modular symbols over imaginary quadratic fields. The results of Merel on Manin symbols have recently been partially generalized by Mohamed \cite{moh} to imaginary quadratic fields $K$, most explicitly in the Euclidean cases, but fell short of obtaining the Fourier expansion in terms of Bianchi modular symbols. We also note that related results were obtained in the thesis of Powell in the context of the Eisenstein ideal and conjectures of Sharifi \cite{powell}. 

We extend Merel's work to Bianchi modular forms over Euclidean imaginary quadratic fields $K$, namely $K = \Q(\sqrt{-d})$ with $d = 1,2,3,7,11$. The Manin symbols in this case were first studied in the early work of Cremona \cite{C}. In principle our work can be extended to non-Euclidean $K$ with class number one, in which case one requires the more complicated pseudo-Euclidean algorithms used in \cite{whit,moh}, but new ideas will be certainly needed for the general case. Our work largely follows the method of Merel, studying the Hecke action on Bianchi modular symbols. As current developments on Bianchi modular forms explore computational approaches such as in  \cite{CA,berger,rahm,rahm2}, we hope that the present study will be a modest contribution towards this body of work. 

To state the main result, we first introduce some notation. Let $G=SL(2)$, and let $\n$ be an ideal of $\O$, where $\O$ is the ring of integers of $K$, a Euclidean imaginary quadratic field. Also let $d_k$ be the discriminant of $K$. Define the congruence subgroups of level $\n$ to be 
\[
\Gamma_1(\n) = \left\{ \begin{pmatrix}a&b\\c&d\end{pmatrix} \in G(\O) : c \equiv d-1 \equiv 0 \pmod \n\right\}
\]
and
\[
\Gamma_0(\n) = \left\{ \begin{pmatrix}a&b\\c&d\end{pmatrix} \in G(\O) : c \equiv  0 \pmod \n\right\}.
\]
Also define
\[
E_\n = \{(u,v)\in(\O/\n)^2 : \langle u , v \rangle = \O/\n\},
\]
which is in bijection with $\Gamma_1(\n)\backslash G(\O)$. 

For any integer $k\ge2$ and $\O$-algebra $R$,  let $R_{k-2}[X,Y]$ be the vector space of homogeneous polynomials of degree $k-2$ in the variables $X,Y$. Also, let $R_{k-2}[X,Y][(\O/\n\O)^2]$ be the space of linear combinations of elements in $(\O/\n\O)^2$ with coefficients in $R_{k-2}[X,Y]$. Given a linear map $\phi$ on this space, we define a linear map given by the formula 
\[
\phi|_g(P(X,Y)[u,v]) = \phi(P(aX+bY,cX+dY)[au+cv,bu+dv])
\]
for $P\in R_{k-2}[X,Y], (u,v) \in (\O/\n\O)^2,$ and $g = (\begin{smallmatrix}a&b\\c&d\end{smallmatrix})$. Define
\[
P_\n = \bigcup_{\m |\n} (\O/\m \O)^\times 
\]
with the convention that for $\m = 1$, $(\O/\m \O)^\times $ is a singleton set. Let $R[P_\n]_k$ be the quotient of $R[P_\n]$ by elements of the form $[a] - (-1)^k[-a]=0,$ for $a\in (\O/\m\O)^\times, \m|\n$. If $a\in (\O/\n\O)^\times$ is invertible mod $\m$ for some $\m|\n$, we denote by $[a]_\m$ the image of $[a \mod \m]$ in $R[P_\n]_k$. Finally, define an $R$-bilinear map 
\begin{align*}
b:R_{k-2}[X,Y][E_\n]&\to R[P_\n]_k \\
P(X,Y)[u,v] &\mapsto P(1,0)[v^{-1}]_{(u,\n)} - P(0,1)[u^{-1}]_{(v,\n)}
\end{align*}
where $(u,\n)$ is the gcd of $u$ and $\n$, which is the same as the order of the subgroup of $\O/\n\O$ generated by $u$, and $v^{-1}$ is identified with the inverse of $v$ mod $(u,\n)$.

Cremona has explicitly identified the relations satisfied by Manin symbols in the form of a certain ideal $\R$ \cite{C}. For example, in the case of $d=1$ the ideal is described in \eqref{R1}. In this case, we shall say $\phi$ satisfies relations determined by $\R$ if 
\be
\label{relation}
\phi + \phi|_S = \phi + \phi|_{ST}+ \phi|_{(ST)^2} = \phi - \phi|_J =  \phi + \phi|_X+ \phi|_{X^2} = 0 ,
\ee
where the matrices $S,T,J,$ and $X$ are defined in \eqref{R0} and \eqref{R1}, and similarly for the other Euclidean fields. 

Our main result then is the following. 

\begin{theorem}
\label{main}
Let $\phi$ be a linear map on $R_{k-2}[X,Y][(\O/\n\O)^2]$ that satisfies relations defined by $\R$, and such that $\phi(P[u,v])=0$ if $[u,v]\not\in E_\n$. Assume moreover the nondeneracy of the pairing of Bianchi cusp forms and cuspidal modular symbols in \eqref{ass}. 

Then for any $x\in R_{k-2}[X,Y][E_\n]$ satisfying $b(x)=0$, the series 
\[
F(z,t) = \sum_{M\in \X}\phi|_M(x) t^2{\bf K}\left(\dfrac{4\pi|\det(M)|t }{\sqrt{d_K}}\right)\psi\left(\dfrac{\det(M) z}{\sqrt{d_K}}\right)
\]
is the Fourier expansion of a Bianchi cusp form of weight $2$ and level $\Gamma_1(\n)$, and $\X$ is the set of Heilbronn-Merel matrices defined in \eqref{HB}. Here  $\psi(z) = e^{2\pi(z+\bar z)}$, ${\bf K}(t) = (-\dfrac{i}{2}K_1(t),K_0(t),\dfrac{i}{2}K_{1}(t)))$ with $K_0,K_1$ the modified Bessel functions, and $(z,t)\in\H_3$. Moreover, all such Bianchi cusp forms can be expressed in this manner.
\end{theorem}
If $b(x)$ is nonzero, then the series obtained should be, up to the constant term, an element of $\SS_k(\n)$. While the theorem is proved for weight 2 forms, we have developed the theory for general weight $k$ where possible. Regarding the assumption of nondegeneracy, we note that this is proved in the classical case by Shokurov \cite{shok} using the cohomology of modular Kuga varieties. In our case, it seems reasonable that the assumption can be removed at least in the case of weight $k=2$.

This paper is organized as follows. In Section \ref{B}, we recall the Bianchi modular forms and Bianchi modular symbols, and prove the properties regarding Manin symbols which we shall require. In Section \ref{C}, we develop the Hecke action on Manin symbols, using the formalism of $C_\Delta$-families and the Heilbronn-Merel matrices. Finally, in Section \ref{U} we prove the main theorem.

\section{Bianchi modular forms and modular symbols}
\label{B}

\subsection{Bianchi modular forms}
For any $\gamma \in G(K)$ and $w=(z,t) \in \H_3$, the upper-half space, define
\[
j(\gamma,w) = \begin{pmatrix}cz + d& - ct \\ \bar{c}t& \overline{cz+d}\end{pmatrix}, \qquad \gamma = \begin{pmatrix}a&b\\ c&d\end{pmatrix}.
\]
Then given a function $F$ on $\H_3$ with values in $\C^3$, define the slash operator
\[
(F|_ k \gamma)(w) = \sigma^k(j(\gamma,w)^{-1})F(\gamma w),
\]
where $\sigma^k$ is the symmetric $k$-th power representation of $G$ on $\C^2$. We define an automorphic form, or Bianchi modular form of weight $k$ on $\Gamma$ to be a harmonic function $F$ such that 
\[
(F|_k \gamma)(w) = F(w)
\]
for all $\gamma\in \Gamma$.  Moreover, we call $F$ a Bianchi cusp form if the integral 
\[
\int_{\C/\O}F|_\gamma(x,y) dx 
\]
vanishes for every $\gamma\in G(\O)$. Denote by $S_k(\Gamma)$ the space of weight $k$ Bianchi cusp forms on $\Gamma$. 

When $k=2$, the $\Gamma$-invariance implies that $F$ has a Fourier-Bessel expansion of the form
 \[
 F(x,y) = \sum_{\alpha\in\O,\alpha\neq 0}c(\alpha) y^2{\bf K}\left(\dfrac{4\pi|\alpha|y}{\sqrt{d_K}}\right)\psi\left(\dfrac{\alpha x}{\sqrt{d_K}}\right)
 \]
where $\psi(x) = e^{2\pi(x+\bar x)}$ and ${\bf K}(t) = (-\dfrac{i}{2}K_1(t),K_0(t),\dfrac{i}{2}K_{1}(t)))$ with $K_0,K_1$ being the modified Bessel functions satisfying the differential equation
\[
\dfrac{dK_j}{dt^2}+\dfrac{1}{t}\dfrac{dK_j}{dy}-\left(1+\dfrac{1}{t^{2j}}\right(K_j = 0, \qquad j=0,1
\]
decreasing rapidly at infinity \cite{sengun}. 

\subsection{Cohomology}
Let $\Gamma$ be a finite-index subgroup of $G(\O)$, and $V$ a representation of $\Gamma$. Then $V$ gives rise to a locally-constant sheaf on the quotient $Y_\Gamma = \Gamma\backslash \H_3$, and it follows from the contractibility of $\H_3$ that 
\[
H^*(\Gamma,V) \simeq H^*(Y_\Gamma,\V), \quad H_*(\Gamma,V) \simeq H_*(Y_\Gamma,\V).
\]
Let $X_\Gamma$ be the Borel-Serre compactification of $Y_\Gamma$, and let $\bar\V$ be a suitable extension of $\V$ to $X_\Gamma$. Define the cuspidal cohomology $H^*_\text{cusp}(Y_\Gamma, \V)$, and thus also $H^*_\text{cusp}(\Gamma, V)$, to be the kernel of the restriction homomorphism
\[
H^*(X_\Gamma,\bar\V) \to H^*(\partial X_\Gamma , \bar\V).
\]

For positive integers $k$ and $l$, consider the representation
\[
V_{k,l}(R) = R_k[X,Y]\otimes R_l[X,Y]^\tau
\]
where $\tau$ is the nontrivial automorphism of $K$, and $R_k[X,Y]$ is the space of homogenous polyonomials of degree $k$ in the variables $X,Y$ with coefficients in $R$, and the action of $G(K)$ is given by the rule 
\begin{align*}
&(P(X,Y)\otimes P'(X,Y))|_g \\
&=  P(dX-bY, -cX+aY)\otimes P(d^\tau X+(-b)^\tau Y, (-c)^\tau X+a^\tau Y), \quad g=\begin{pmatrix}a & b\\ c&d\end{pmatrix}.
\end{align*}
Then a well-known result of Harder \cite{harder} gives
\be
\label{harder}
H^i_\text{cusp}(\Gamma,V_{k,k}(\C))\simeq S_k(\Gamma)
\ee
for $i=1,2$. 

\subsection{Bianchi modular symbols}

Let $\Gamma$ be a finite index subgroup of $G(\O)$. Let $R$ be an $\O$-module, and $k\ge2$ an integer; if $k$ is odd, further suppose that $-\text{Id}\not\in\Gamma$. Let $\M$ be the torsion free abelian group generated by pairs $\{\alpha,\beta\}$, where $\alpha,\beta\in\P^1(K)$, modulo the relations
\[
\{\alpha,\alpha\} = \{\alpha,\beta\} + \{\beta,\gamma\}+\{\gamma,\alpha\} =0
\]
for any $\alpha,\beta,\gamma\in\P^1(K)$. Then define
\[
\M_k =R_{k-2}[X,Y]\otimes \M,
\]
where $R_{k-2}[X,Y]$ is the space of complex homogeneous polynomials in the variables $X$ and $Y$ of degree $k-2$ with coefficients in $R$. Clearly $\M_2 = \M$. Define a linear action of $G(K)$ on $P\in R_{k-2}[X,Y]$ and $P\otimes\{\alpha,\beta\}\in\M_k$ by
\[
P|_g(X,Y) = P(dX-bY, -cX+aY),\qquad g=\begin{pmatrix}a & b\\ c&d\end{pmatrix}
\]
and 
\[
P\otimes\{\alpha,\beta\}|_g = P|_g\otimes \{g\alpha,g\beta\}.
\]
Let $\M_k(\Gamma)$ be the quotient of $\M_k$ by the relation $x|_\gamma = x$ for all $x\in\M_k, \gamma\in\Gamma.$ We denote by $P\{\alpha,\beta\}$ the image of $P\otimes \{\alpha,\beta\}\in\M_k$ in $\M_k(\Gamma)$. The elements of $\M_k(\Gamma)$ will be called modular symbols of weight $k$ for $\Gamma$. 

It follows from \cite[Theorem 4.1]{moh} and Poincar\'e duality, that if $\Gamma$ is a congruence subgroup and its torsion elements have orders invertible in $R$, one has
\be
\label{moh}
H_1 (\Gamma, V_{k,k}(\C) ) \simeq \M_{k}(\Gamma)
\ee
for $k\ge2$.

\subsection{Boundary modular symbols}

Let $\B$ be the abelian group generated by the elements $\{\alpha\}, \alpha\in \P^1(K)$, and let $\B_k = R_{k-2}[X,Y]\otimes \B$. Define a linear action of $g\in G$ on $P\otimes\{\alpha\}\in \B_k$ by
\[
P\otimes\{\alpha\}|_g = P|_g\otimes \{g\alpha\}.
\]
We then define $\B_k(\Gamma)$ to be the quotient of $\B_k$ by the relation $x|_\gamma = x$ for all $x\in \B_k,\gamma\in \Gamma$, and we call the elements of $\B_k(\Gamma)$ the boundary modular symbols of weight $k$ of level $\Gamma$.

Let $R[\Gamma\backslash K^2]_k$ be the space $R[\Gamma\backslash K^2]$ modulo the relation 
\[
[\Gamma(\lambda u, \lambda v)]\sim \left(\dfrac{\lambda}{|\lambda|}\right)^k [\Gamma(u, v)].
\]
If $k$ is even, this vector space is canonically isomorphic to $R[\Gamma\backslash \P^1(K)]$, while for any $k\ge2$, its dimension is equal to $|\Gamma\backslash\P^1(K)|$. Also define the linear map
\[
\mu: \B_k(\Gamma) \to R[\Gamma\backslash K^2]_k
\]
sending $P\otimes \{\dfrac{u}{v}\}$ to the element $P(u,v)[\Gamma(u,v)]$. 

\begin{lemma}
The map $\mu$ is well-defined, and is an isomorphism.
\end{lemma}
\begin{proof}
We first decompose $\B_k(\Gamma)$ into the direct sum $\oplus_\alpha \B_k(\Gamma)_\alpha$ as $\alpha$ runs over a set of representatives of $\Gamma\backslash \P^1(K)$, and $\B_k(\Gamma)_\alpha$ is the subspace generated by $P\{\alpha\}, P\in R_{k-2}[X,Y]$. For each $\beta\in \P^1(K)$, we have a surjective linear map
\[
\psi_\beta: R_{k-2}[X,Y] \to \B_k(\Gamma)
\]
sending $P$ to $P\{\beta\}$, whose kernel is generated by polynomials of the form $P - P|_{g_\beta}$, where $g_\beta\in\Gamma$ satisfies $g_\beta\beta= \beta$. It follows then that $(P-P|_{g_\beta})(\beta)=0$, and $\mu$ is well-defined.

On the other hand, since the kernel of $\psi_\beta$ contains the kernel of a linear map, and $\psi_\beta$ is nontrivial for any given $\beta$, it follows that the image of $\psi_\beta$ is one-dimensional. The dimension of $\B_k(\Gamma)$ then is equal to $|\Gamma\backslash\P^1(K)|$, hence $\mu$ is bijective.
\end{proof}

Define the boundary map
\[
\partial: P\otimes \{\alpha,\beta\} \to P\otimes \{\beta\} - P\otimes \{\alpha\},
\]
and define $\SS_k(\Gamma)$ to be the kernel of $\partial$. We shall call $\SS_k(\Gamma)$ the space of cuspidal modular symbols. Also define $\theta$ to be the linear map sending $\lambda\{\alpha\}$ to the element $\lambda$.

\begin{proposition}
\label{exact}
There are exact sequences 
\[
0 \to \SS_k(\Gamma) \to \M_k(\Gamma) \stackrel{\partial}{ \to} \B_k(\Gamma) \to 0
\]
if $k>2$, and
\[
0 \to \SS_k(\Gamma) \to \M_k(\Gamma) \stackrel{\partial}{ \to} \B_k(\Gamma) \stackrel{\theta}{ \to} R \to 0
\]
if $k=2$.
\end{proposition}
\begin{proof}
This follows from a simple adaptation of the arguments of \cite[Proposition 5]{merel} and \cite[Lemma 2.3]{moh}.
\end{proof}

We shall also require the following formula.
\begin{lemma}
\label{boundary}
For any $P\in R_{k-2}[X,Y], g\in G(\O)$, the identity
\be
\mu(\partial([P,g])) = P(1,0)[\Gamma g(1,0)] - P(0,1)[\Gamma g(0,1)]
\ee
holds in $R[\Gamma\backslash K^2]_k$.
\end{lemma}
\begin{proof}
This follows from the property
\begin{align*}
\partial([P,g]) &= \partial\left(P|_g\otimes \left\{\dfrac{b}{d},\dfrac{a}{c}\right\}\right) \\
&= P|_g\otimes \left\{\dfrac{a}{c}\right\} - P|_g\otimes \left\{\dfrac{b}{d}\right\}, \quad g = \begin{pmatrix} a & b \\ c & d \end{pmatrix}
\end{align*}
and the definition of $\mu$.
\end{proof}

\subsection{Manin symbols}

Let $Y^*_\Gamma= \Gamma\backslash \H_3^*$ be the Satake compactification of $Y_\Gamma$, where $\H_3^* = \H_3 \cup \P^1(K)$. The modular symbol $\{\alpha,\beta\}$ can be viewed as the homology class of the relative cycle induced by the geodesic joining the cusps $\alpha$ and $\beta$ in $Y_\Gamma^*$. Choose a finite covering of $Y^*_\Gamma$ by the images of open principal hemispheres, and let $\CC = \{a_1,\dots,a_n\}$ be the centers of these hemispheres. Define the Manin symbol with respect to $a\in\CC$ to be
\[
[P,g,a]=P|_g\otimes \{ga,g\infty\}\in\M_k(\Gamma).
\]
Since $K$ is Euclidean, one can choose $\CC=\{0\}$, and we simply write $[P,g] = [P,g,0]$.

Following \cite[Theorem 2]{C}, the Manin symbols can be used to generate the homology group $H_1(Y^*_\Gamma,\Q)$ modulo relations determined by a certain relational ideal $\R$. If we set 
\be\label{R0}
I = \begin{pmatrix}1  &0\\0&1 \end{pmatrix}, \ J = \begin{pmatrix}\varepsilon &0\\0&1 \end{pmatrix}, \ S = \begin{pmatrix}0&-1\\1&0 \end{pmatrix},\  TS = \begin{pmatrix}1&-1\\1&0 \end{pmatrix}, 
\ee
where $\varepsilon$ is the fundamental unit of $K$, then for each Euclidean $K$ the ideal $\R$ contains the relations
\[
\R_0 = \langle I+S, I-J, I + (TS) + (TS)^2\rangle
\]
and the remaining relations are detailed in \cite[p.295]{C}. For example, if $d=1$, one has
\be
\label{R1}
\R = \R_0 + \langle I + X + X^2\rangle, \quad X =  \begin{pmatrix}i &1 \\1&0 \end{pmatrix}
\ee
where $i=\sqrt{-1}$.

\begin{proposition}
The Manin symbols generate  $\M_k(\Gamma)$, and satisfy relations determined by $\R$, for example, the ideal $\R_0$ gives
\[
[P,g] + [P|_{S^{-1}}, gS] = 0, \quad [P,g] = [P|_{J^{-1}},  gJ] ,
\]
\[
[P,g] + [P|_{(TS)^{-1}},g(TS)] + [P|_{(TS)^{-2}},g(TS)^2] = 0.
\]
\end{proposition}

\begin{proof}
The first assertion follows from the Euclidean algorithm for $\O$ as demonstrated in the proof of \cite[Theorem 2]{C} (see also \cite[Theorem 3.5.7]{powell}), generalizing Manin's trick using the continued fraction expansion to this setting. In the second case, the relations follow from the properties of Manin symbols such as in the proof of \cite[Proposition 1]{merel}. We demonstrate this with the first formula. One checks that
\begin{align*}
[P,g] + [P|_{S^{-1}}, gS] &= P|_g \otimes \{g0,g\infty\} + (P|_{S^{-1}})|_{gS} \otimes \{gS0,gS\infty\}\\
&= P|_g \otimes \{g0,g\infty\} + P|_{gSS^{-1}} \otimes \{g\infty, g0\}\\
&=0,
\end{align*}
since $\infty = S0 = S^2\infty$ in $\P^1(K)$. 
\end{proof}

Note that if $\Gamma g_1 = \Gamma g_2$, then $[P,g_1]= [P,g_2]$ since the symbol is invariant by the left action of $\Gamma$. And since $\Gamma$ is a subgroup of finite index, the abelian group generated by Manin symbols is also of finite rank, generated by
\[
[X^{k-2-i} Y^i, g_j], \qquad 0\le i \le k-2,
\]
and $g_j$ runs through representatives of right cosets $\Gamma\backslash G(\O)$.  It follows then that $\M_k(\Gamma)$ is a finite dimensional vector space, and there is a natural map
\[
R_{k-2}[X,Y][\Gamma\backslash G(\O)]\to \M_k(\Gamma)
\]
sending the element $P[\Gamma g]$ to the Manin symbol $[P,g].$

\subsection{Pairing}

Let $\overline{S_k(\Gamma)} $ be the space spanned by $\overline{f(w)} $ for  each $f\in S_k(\Gamma)$. By \eqref{harder} and \eqref{moh}, there is a well-defined pairing  \[
(S_k(\Gamma) \oplus \overline{S_k(\Gamma)}) \times \M_k(\Gamma) \to \C
\]
 induced by the evaluation pairing. More precisely, define the pairing 
given by
\be
\label{deg}
\langle f_1 + f_2 , P\otimes \{\alpha,\beta\}\rangle = \int_\alpha^\beta f_1(w) P(w,1) dw + \int_\alpha^\beta f_2(w) P(\bar{w},1) d\bar{w}.
\ee
The map is in general degenerate, and the elements $m \in\M_k(\Gamma)$ such that $\langle f,m\rangle=0$ for all $f\in (S_k(\Gamma) \oplus \overline{S_k(\Gamma)}) $ are called Eisenstein elements. 

We then require the following assumption: The pairing \eqref{deg}  when restricted to
\be
\label{ass}
(S_k(\Gamma) \oplus \overline{S_k(\Gamma)})  \times \SS_k(\Gamma) \to \C
\ee
is nondegenerate. In the classical case, this is proved using the theory of Shokurov symbols using the cohomology of certain modular Kuga varieties \cite[Theorem 0.2]{shok}.

\section{Computing the Hecke action}
\label{C}

\subsection{$C_\Delta$ families}

Let $\Gamma$ be a congruence subgroup of $G$, and $\Delta$ a subgroup of Mat$_2(\O)$ such that $\Gamma\Delta = \Delta\Gamma = \Delta$ and $\Gamma\backslash\Delta$ is finite. Fix a finite set $R$ of representatives for $\Gamma\backslash\Delta$. There is a well-defined linear map
$T_\Delta: \M_k(\Gamma) \to \M_k(\Gamma)$ given by 
\[
T_\Delta: P\otimes \{\alpha,\beta\} \mapsto  \sum_{\delta\in R}  P|_\delta \otimes \{\delta\alpha,\delta\beta\}.
\]
It is independent of the choice of representatives $R$, and it can be shown in fact that $T_\Delta$ maps $\SS_k(\Gamma)$ to itself. Also define the Shimura involution 
\[
g = \begin{pmatrix} a & b\\ c & d\end{pmatrix} 
\mapsto 
 \tilde{g} = \begin{pmatrix} d & -b\\ -c & a\end{pmatrix} = g^{-1}\det(g)
\]
and denote $\tilde{\Delta}=\{g \in GL_2(K) : \tilde{g}\in \Delta\}$. Suppose there exists a map 
\[
\phi: \tilde{\Delta}G(\O)\to G(\O)
\]
such that the pair $(\phi,\Delta)$ satisfies the following conditions:
\begin{enumerate}
\item 
for all $\gamma\in\tilde{\Delta}G(\O)$ and $g\in G(\O)$, we have $\Gamma\phi_\Delta(\gamma g) = \Gamma\phi_\Delta(\gamma)g,$\\
\item 
for all $\gamma\in\tilde{\Delta}G(\O)$, we have $\gamma\phi_\Delta(\gamma)^{-1}\in\tilde{\Delta},$ or equivalently, $\phi(\gamma)\tilde{\gamma}\in\Delta,$\\
\item 
the map $\Gamma\backslash\Delta \to\tilde{\Delta}G(\O)/G(\O)$ given by $\Gamma \delta \mapsto \tilde{\delta} G(\O)$ is injective. Note that it is also necessarily surjective.
\end{enumerate}
These conditions ensure that summing over $g\tilde{\Delta}G(\O)$ is the same as summing over the quotient $\Gamma\backslash \Delta$.

We say that an element 
\[
\sum_M u_M M \in R[\text{Mat}_2(\O)].
\]
satisfies condition $(C_\Delta)$ if and only if for all $C\in\tilde{\Delta}G(\O)/G(\O)$ we have the equality in $R[\P^1(K)]$
\[
[\infty] - [0]=\sum_{M\in C}u_M ([M\infty] - [M0]) .
\]
We then have the following formula for the Hecke action on Manin symbols. 
\begin{proposition}
\label{merel4}
Let $P\in R_{k-2}[X,Y]$ and $g\in G(\O)$. Given an element $\sum_M u_M M\in R[\mathrm{Mat}_2(\O)]$ satisfying condition $(C_\Delta)$, and a $(\phi,\Delta)$ pair as above, we have an equality in $\M_k(\Gamma)$,
\[
T_\Delta([P,g]) = \sum_M u_M[P|_{\tilde{M}},\phi(gM)]
\]
where the sum runs over $M$ such that $gM\in \tilde{\Delta}G(\O)$.
\end{proposition}
\begin{proof}
Let $S$ be a set of representatives of $g^{-1}\tilde{\Delta}G(\O)/G(\O)$. Then the right hand side can be written as 
\begin{align*}
&\sum_{s\in S}\sum_{M\in sG(\O)} u_M P|_{\phi(gM)\tilde{M}}\otimes \{\phi(gM)0,\phi(gM)\infty\}\\
&= \sum_{s\in S}\sum_{M\in sG(\O)} u_M P|_{\phi(gs)s^{-1}M\tilde{M}}\otimes \{\phi(gs)s^{-1}M0,\phi(gs)s^{-1}M\infty\}
\end{align*}
since $s^{-1}M\in G(\O)$ and property (1) of the $(\phi,\Delta)$ family. Using condition $(C_\Delta)$ the definition of modular symbols this is equal to 
\[
\sum_{s\in S}u_M P|_{\phi(gs)s^{-1}M\tilde{M}}\otimes \{\phi(gs)s^{-1}0,\phi(gs)s^{-1}\infty\}.
\]
Since $\det(s)=\det(M),$ we have $s^{-1}M\tilde{M}=\tilde{s},$  thus we write the above as 
\[
\sum_{s\in S}u_M P|_{\phi(gs)\tilde{s}\tilde{g}g}\otimes \{\phi(gs)\tilde{s}\tilde{g}g0,\phi(gs)\tilde{s}\tilde{g}g\infty\}.
\]
By property (2) of the map $\phi$ we have $\phi(gs)\tilde{s}\tilde{g}\in\Delta,$ and by property (3) we have that $\phi(gs)\tilde{s}\tilde{g}$ runs through a set of representatives of $\Gamma\backslash\Delta$ as $s$ runs through $S$. Thus we conclude that the above is equal to 
\[
\sum_{\delta\in R} \delta[P,g]
\]
as desired.
\end{proof}

We refer the reader to \cite[Theorem 3.2]{moh} for the analogous result.

\subsection{Heilbronn-Merel families}

We recall the Heilbronn-Merel families constructed by Mohamed that satisfy condition $C_\Delta$ \cite[Section 3.3]{moh}. Since we are assuming $K$ is Euclidean, it is constructed as a subset of 
\[
\X_\eta \subset  \left\{ \begin{pmatrix}a& b\\ c & d\end{pmatrix}: N(a) > N(b) \ge0, N(d) > N(c) \ge 0, ad - bc = \eta\right\}.
\]
Let $\D$ be the set of divisors of $\eta$, where we identify two divisors $\delta$ and $\delta'$ if they are associate, that is, $\delta,\delta'$ are distinct in $\D$ if $\delta/\delta'\not\in\O^\times.$ Then given $\delta\in\D$, choose a set of representatives $S_\delta$ of $\O/\delta\O$ such that for each $\beta\in S_\delta$, we have $N(\beta)<N(\delta)$. Given a set of representatives $A_\delta$ of $\O/\delta\O$, for each $\alpha\in A_\delta$ let $\alpha'$ be the remainder of the division of $\alpha$ by $\delta$. Then the set of $\alpha'$ satisfies the desired property of $S_\delta$.

The Heilbronn-Merel matrices of determinant $\eta$ are then constructed as follows. Given $\beta\in S_\delta$, let $x_0=\delta, x_1=\beta, y_0=0,$ and $y_1=\eta/\delta$. Form the matrices
\[
M_1 = \begin{pmatrix} x_0& x_1\\ y_0& y_1\end{pmatrix} , \quad M_1 = \begin{pmatrix} x_1& x_2\\ y_1& y_2\end{pmatrix} 
\]
where 
\[
x_2 = x_1q_1 - x_0, \quad y_2 = y_1 q_1 - y_0.
\]
Here $-x_2$ is a remainder obtained from the division of $x_0$ by $x_1$. Both $M_1,M_2$ have determinant $\eta$, and $N(x_2)\le \epsilon N(x_1)$, where
\[
\epsilon = \begin{cases}\left(\dfrac{1+d}{4\sqrt{d}}\right)^2 & \text{if } d\equiv 3 \bmod 4 \\ \dfrac{1+d}{4}&\text{otherwise}.\end{cases}
\]
and $K=\Q(\sqrt{-d})$. More generally, from a matrix 
\[
M_i = \begin{pmatrix}x_{i-1}&x_i\\y_{i-1}&y_i\end{pmatrix},\quad x_i=x_{i-1}q_{i-1} - x_{i-2}, y_i=y_{i-1}q_{i-1} - y_{i-2}
\]
we form the matrix
\[
M_{i+1} = \begin{pmatrix}x_{i}&x_{i+1}\\y_{i}&y_{i+1}\end{pmatrix},\quad x_{i+1}=x_{i}q_{i} - x_{i-1}, y_{i+1}=y_{i}q_{i} - y_{i-1},
\]
and stop once the remainder is zero. The matrices $M_{i}$ and $M_{i+1}$ are more succinctly related by the equality 
\[M_{i+1}= M_i\begin{pmatrix}0&-1\\1&q_1\end{pmatrix},
\]
so we may describe the Heilbronn-Merel families as the sets 
\begin{align}
\X_{\eta} = \bigcup_{\delta\in\D}\bigcup_{\beta\in S_\delta} \Big\{ M_i = \begin{pmatrix}x_{i-1}&x_i\\y_{i-1}&y_i\end{pmatrix}: M_{i+1}&= M_i\begin{pmatrix}0&-1\\1&q_1\end{pmatrix}, M_1 = \begin{pmatrix} x_0& x_1\\ y_0& y_1\end{pmatrix},\notag\\ 
\text{and } x_{i+1}&= x_iq_i - x_{i-1}\Big\}\label{HB}
\end{align}
and  define $\X$ to be the union over all nonzero $\eta \in \O$ of $\X_\eta.$ Any element $(\begin{smallmatrix} a& b \\ c & d\end{smallmatrix}) \in \X_\eta$ satisfies $N(a)> N(b) \ge 0$ and $N(d) > N(c) \ge 0$, analogous to the Heilbronn families used by Merel \cite{merel}. By construction, the element
\[
\sum_{\theta \in \X_\eta} \theta\in R[\text{Mat}_2(\O)]
\]
satisfies condition $(C_\Delta)$. We remark that an explicit algorithm producing Heilbronn-Merel families is given in Algorithm 3.17 of \cite{moh}. In the case that $K$ is non-Euclidean but has class number one, one defines a set of matrices defining $T_{\Delta_\eta}$ under the assumption that a word decomposition of an element $g\in G$ is available.

\subsection{Hecke action}

It remains to construct a pair $(\phi,\Delta)$. In the case of $\Gamma = \Gamma_1(\n)$, it is constructed as follows. Let $\eta\in\O$ be a nonzero element coprime to $\n$. Define
\[
\Delta_\eta = \left\{\begin{pmatrix}a& b\\ c & d\end{pmatrix}\in\mathrm{Mat}_2(\O): ad-bc=\eta, c\equiv a-1\equiv 0\  (\bmod \n)\right\}.
\]
If  $M_2(\O)_\eta$ is the subset of elements of $\mathrm{Mat}_2(\O)$ with determinant $\eta$, then $\Delta_\eta = M_2(\O)_\eta \cap \Gamma_1(\n)$.   
It follows that $\Gamma_1(\n)\Delta_\eta = \Delta_\eta\Gamma_1(\n)= \Delta_\eta$ and the set $\Gamma_1(\n) \backslash \Delta_\eta$ is finite. We denote then $T_{\Delta_\eta}$ the Hecke operator acting on $\M_k(\n)$. 

Define 
\[
E_\n=\{(u,v)\in (\O/ \n\O)^2:u\O+v\O = \O/\n\O\}.
\]
There is a surjective map $\pi:G(\O)\to E_\n$ sending $M$ to $(0,1)M$. Note that $M,M'\in G(\O)$ have the same image under this map if and only if $M\in \Gamma_1(\n) M'$, therefore we have a bijection
\[
\Gamma_1(\n)\backslash G(\O)\stackrel{\sim}{\to} E_\n.
\]
Let $\lambda$ be a section of $\pi$, and let $x\in E_\n$. Then the Manin symbol $[P,\lambda(x)]$ depends only on $\Gamma_1(\n)\lambda(x)$ and $P$. Define an action of Mat$_2(\O)$ on $x=(u,v)$ by the matrix mutiplication
\[
(u,v)\begin{pmatrix}a  & b\\c & d\end{pmatrix} = (au+cv, bu+dv).
\]
The space of Manin symbols of weight $k$ and level $\Gamma_1(\n)$ over a ring $R$ can then be identified with $R_{k-2}[X,Y][E_\n]$. That $(\phi_\eta,\Delta_\eta)$ is a $(\phi,\Delta)$ pair is proven in \cite[Lemma 3.3]{moh}, we nonetheless include a direct proof of this in the proof of the following proposition.

\begin{theorem}
\label{heckethm}
Let $P[u,v]\in R_{k-2}[X,Y][E_\n]$. Let $\sum_Mu_M\in R[M_2(\O)_\eta]$ be such that for all classes $C\in M_2(\O)_\eta/G(\O)$, we have
\be
\label{Cn}
\sum_{M\in C}u_M([M\infty]-[M0])= [\infty] - [0].
\ee
Then
\[
T_{\Delta_\eta}([P,(u,v)]) = \sum_{M\in C} u_M[P(aX+bY,cX + dY), (au+cv,bu+dv)]
\]
where the sum runs over matrices $M$ such that $(au+cv,bu+dv)\in E_\n$.
\end{theorem}

\begin{proof}
We first note that the requirement \eqref{Cn} implies that the sum $\sum_Mu_M$ satisfies condition $C_{\Delta_\eta}$. On the other hand, observe that the set $\tilde{\Delta}_\eta G(\O)$ is the set of matrices in $M_2(\O)_\eta$ such that 
\[
c\O/\n\O + d\O/\n\O = \O/\n\O,
\]
in other words, $c,d,\n$ are coprime ideals. Define $\phi_\eta:\tilde{\Delta}_\eta G(\O) \to G(\O)$ to be the map such that
\[
\pi(\phi_\eta(\begin{pmatrix}a  & b\\c & d\end{pmatrix})) = (c,d) \in E_\n.
\]
In particular, we have that $\pi(\phi_\eta(M))= (0,1)M$ for any $M$. Also, note that $gM\in \tilde{\Delta}_\eta G(\O)$ if and only if $\pi(g)M\in E_\n.$

If we can show that $(\phi_\eta,\Delta_\eta)$ is a $(\phi,\Delta)$ pair, then Proposition \ref{merel4} will apply and the result will follow. 

Property (1) follows since
\[
\pi(\phi_\eta(\begin{pmatrix}a  & b\\c & d\end{pmatrix}g)) = (c,d)g  = \pi(\phi_\eta(\begin{pmatrix}a  & b\\c & d\end{pmatrix})g).
\]

Property (2) follows since $g\in \tilde{\Delta}_\eta G(\O)$ if and only if $(0,1)g = (0,1)$ in $E_\n$, so
$
(0,1)g = (0,1)\phi_\eta(g) 
$
and thus $g\phi_\eta(g)^{-1}\in\tilde{\Delta}_\eta$. 

For Property (3), we consider $\delta,\delta'\in \Delta_\eta$ such that
\[
\delta' \delta^{-1} = \begin{pmatrix}a'  & b'\\c' & d'\end{pmatrix}\begin{pmatrix}a  & b\\c & d\end{pmatrix}^{-1} 
=\eta^{-1} \begin{pmatrix} da'-b'c & -ba'+ab'\\dc'-d'c & -bc'+ad'\end{pmatrix}
\]
belongs to $G(\O).$ By definition of $\Delta_\eta$ it follows that $\eta \equiv d \equiv d' \pmod \n$, so $\delta,\delta',$ and $\delta'\delta^{-1}$ are all upper triangular matrices modulo $\n$. Moreover, since $a\equiv a'\equiv 1\pmod \n,$ it follows that $da'-b'c\equiv 1\pmod \n$ also, thus $\delta'\delta^{-1}$ belongs to $\Gamma_1(\n)$.
\end{proof}

\section{Universal expansions}
\label{U}

\subsection{A Fourier expansion}

Define the Hecke algebra $\T$ to be the commutative subalgebra of the endomorphisms of $H^1(\Gamma,V)$, generated by elements 
\[
T_\pi = \begin{pmatrix}\pi&0\\ 0&1\end{pmatrix},\quad S_\pi= \begin{pmatrix}\pi&0\\ 0& \pi\end{pmatrix}
\]
in $GL_2(K)$, where $\pi$ is a prime in $\O$ coprime to the level of $\Gamma$. It is well-known that $\T$ stabilizes the the cuspidal part $H^i_\text{cusp}(\Gamma,V)$.

For ease of notation, we shall set $S_k(\n) = S_k(\Gamma_1(\n))$ and similarly $\SS_k(\n)$ and $\M_k(\n)$.

\begin{proposition}
\label{linearheck}
Let $\Lambda$ be a linear map from $\T$ to $\C$. Then for $(z,t)\in \H_3$, 
\be
\label{linear}
 \sum_{\alpha\in\O,\alpha\neq 0}\Lambda(T_\alpha) t^2{\bf K}\left(\dfrac{4\pi|\alpha|t }{\sqrt{d_K}}\right)\psi\left(\dfrac{\alpha z}{\sqrt{d_K}}\right)
\ee
is, except for the constant coefficient, the Fourier expansion of an element of $S_2(\n)$.
\end{proposition}

\begin{proof}
Let $a_\alpha$ be the linear form on $M$ which associates to a Bianchi modular form its $\alpha$-th Fourier coefficient . We claim that the bilinear pairing on $S_2(\n)\times \T$ given by
\[
(f,T) \mapsto a_1(Tf)
\]
is nondegenerate.  First, fix $T\in\T$ and suppose that for all $f\in S_2(\n)$ we have $a_1(Tf) = 0$. Then for all nonzero $\alpha\in \O$, we have 
\[
0 = a_1(TT_\alpha f) = a_1(T_\alpha Tf) = a_\alpha(Tf),
\]
and it follows that $Tf=0$ for all $f\in S_2(\n)$, and therefore $T=0$.  Conversely, fix $f\in S_2(\n)$ and suppose that for all $T\in \T$ we have $a_1(Tf) = 0$. Then for all $\alpha\in\O$ we have
\[
a_1(T_\alpha f ) = a_\alpha(f) = 0
\]
and thus $f=0$.

Now then there exists $f\in S_2(\n)$ such that $\Lambda(T) = a_1(Tf)$ for any $T\in\T$. It follows then that \eqref{linear} is equal to 
\begin{align*}
& \sum_{\alpha\in\O,\alpha\neq 0} a_1(T_\alpha f) t^2{\bf K}\left(\dfrac{4\pi|\alpha|t }{\sqrt{d_K}}\right(\psi\left(\dfrac{\alpha z}{\sqrt{d_K}}\right)\\
&=  \sum_{\alpha\in\O,\alpha\neq 0} a_\alpha (f) t^2{\bf K}\left(\dfrac{4\pi|\alpha|t }{\sqrt{d_K}}\right(\psi\left(\dfrac{\alpha z}{\sqrt{d_K}}\right)
\end{align*}
which is the the Fourier expansion of $f$, up to the constant Fourier coefficient.
\end{proof}

\subsection{Proof of main theorem}
We can now complete the proof of Theorem \ref{main}.  Let
\[
x = \sum_{\lambda\in E_\n} P_\lambda[\lambda],\quad m(x)= \sum_{\lambda\in E_\n} [P_\lambda,\lambda]\in \M_k(\n).
\]
Using Lemma \ref{boundary} and the property that $\SS_k(\n)$ lies in the kernel of $\partial$ from Proposition \ref{exact}, it follows that the requirement $b(x) =0$ is equivalent to $m(x) \in\SS_k(\n)$. On the other hand, the relations \eqref{relation} imply that that the linear map $\phi$ factorizes through a linear map on $\M_k(\n)$ by means of the map
\[
R_{k-2}[X,Y][\Gamma\backslash G(\O)] \to \M_k(\n)
\]
sending $P[\Gamma g]$ to the Manin symbol $[P,g]$. Denote by $\phi_m$ the linear map on $\SS_k(\n)$ induced by $\phi$. 

By our assumption that \eqref{ass} is nondegenerate, the Hecke algebra on $S_k(\Gamma)$ is isomorphic to the algebra generated by operators $T_{\Delta_\eta}$ acting on $\M_k(\Gamma)$. The map
$
T \mapsto \phi_m(T(m(x))
$
is a linear map on the Hecke algebra. By Theorem \ref{heckethm}, we have for the $(C_{\Delta_\eta})$ family $\X_\eta$,
\begin{align*}
\phi_m(T_\alpha m(x)) &= \sum_{\lambda\in E_\n} \phi_m\left(\sum_{M\in \X_\eta}[P_\lambda|_{\bar M},\lambda M]\right)\\
&= \sum_{\lambda\in E_\n} \phi \left(\sum_{M\in \X_\eta}P_\lambda|_{\bar M}[\lambda M]\right)\\
&=\sum_{M\in \X_\eta} \phi|_M(x).
\end{align*}
Here we have used the fact that $\phi(P[x])=0$ if $x\not\in E_\n$.  Then taking $k=2$ and applying Proposition \ref{linearheck}, we have that 
\begin{align*}
& \sum_{\alpha\in\O,\alpha\neq 0}\phi_m(T_\alpha m(x)) t^2{\bf K}\left(\dfrac{4\pi|\alpha|t }{\sqrt{d_K}}\right)\psi\left(\dfrac{\alpha z}{\sqrt{d_K}}\right)
\\ 
&=  \sum_{\alpha\in\O,\alpha\neq 0} \sum_{M\in\X_\alpha} \phi|_M(x) t^2{\bf K}\left(\dfrac{4\pi|\alpha|t }{\sqrt{d_K}}\right)\psi\left(\dfrac{\alpha z}{\sqrt{d_K}}\right)\\ 
& = \sum_{M\in \X} \phi|_M(x) t^2{\bf K}\left(\dfrac{4\pi|\det(M)|t }{\sqrt{d_K}}\right)\psi\left(\dfrac{\det(M) z}{\sqrt{d_K}}\right) 
\end{align*}
is the Fourier expansion of a Bianchi cusp form in $S_k(\n)$.

Conversely, given any $f\in S_k(\n)$, let $\phi_f$ be the unique linear form on $\SS_k(\n)$ such that $\phi_f(y) = \langle f,y\rangle$ for all $y\in\SS_k(\n)$. Let now $\phi$ be the composition of $\phi_f$ with the canonical surjection from ker$(b)$ to $\SS_k(\n)$, and choose an element $x\in R_{k-2}[X,Y][\O/\n\O)^2]$ such that $b(x)=0$ and whose image $A$ in $\SS_k(\n)$ is the unique modular symbol such that $\langle f,A\rangle = a_1(f)$. Then since
\[
\phi(T_\alpha(x)) = \phi_f(T_\alpha A) = a_1(T_\alpha f) = a_\alpha(f)  
\]
we then have the Fourier expansion of $f$
\begin{align*}
& \sum_{\alpha\in\O,\alpha\neq 0}\phi(T_\alpha(x)) t^2{\bf K}\left(\dfrac{4\pi|\alpha|t }{\sqrt{d_K}}\right)\psi\left(\dfrac{\alpha z}{\sqrt{d_K}}\right)\\
& =  \sum_{\alpha\in\O,\alpha\neq 0}a_\alpha(f) t^2{\bf K}\left(\dfrac{4\pi|\alpha|t }{\sqrt{d_K}}\right)\psi\left(\dfrac{\alpha z}{\sqrt{d_K}}\right), 
\end{align*}
hence all Bianchi modular forms can be obtained by this method. 

\subsubsection*{Acknowledgments}The author thanks Debargha Banerjee and John Voight  for helpful discussions related to this work.  

\bibliography{HeckeAction}
\bibliographystyle{abbrv}

\end{document}